\documentclass[12pt]{amsart}
\usepackage{amssymb,amsfonts,amsmath,amsopn,amstext,amscd,latexsym,amsthm,enumerate,mathrsfs,bbm,color}

\usepackage[margin=36mm]{geometry}
\usepackage[all,cmtip]{xy}
\usepackage{longtable}

\newtheorem{theorem}{Theorem}[section]
\newtheorem{thm}{Theorem}

\newtheorem{prop}[theorem]{Proposition}

\theoremstyle{remark}

\numberwithin{equation}{section}

\begin{document}

\title[$M$-groups and Codegrees]
{$M$-groups and Codegrees; $M_{p}$-groups and Brauer Character Degrees}

\author[X. Chen]{Xiaoyou Chen}
\address{School of Mathematics and Statistics, Henan University of Technology, Zhengzhou 450001, China}
\email{cxymathematics@hotmail.com}

\author[M. L. Lewis]{Mark L. Lewis}
\address{Department of Mathematical Sciences, Kent State University, Kent, OH 44242, USA}
\email{lewis@math.kent.edu}


\date{\today}


\begin{abstract}
Let $G$ be a finite group and $p$ be a prime.  We prove that if $G$ has three codegrees, then $G$ is an $M$-group.  We prove for some prime $p$ that if every irreducible Brauer character of $G$ is a prime, then for every normal subgroup $N$ of $G$ either $G/N$ or $N$ is an $M_p$-group.
\end{abstract}

\maketitle

\section{Introduction}

All groups in this note are finite, and refer to \cite{Isaacs1976} and \cite{Navarro1998} for notation.  Let $G$ be a group and ${\rm Irr}(G)$ be the set of irreducible (complex) characters of $G$. Recall that a character $\chi$ of $G$ is {\it monomial} when $\chi$ is induced from a linear character of some subgroup of $G$.  If every character $\chi\in {\rm Irr}(G)$ is monomial, then $G$ is said to be an {\it $M$-group}.  A well-known theorem of Taketa states that an $M$-group is solvable (see \cite[Corollary 5.13]{Isaacs1976}).

Let $\chi\in {\rm Irr}(G)$ and write
$${\rm cod}(\chi)=\frac{|G: \ker\chi|}{\chi(1)}.$$
And write ${\rm cod}(G)=\{{\rm cod}(\chi)\mid \chi\in {\rm Irr}(G)\}$.

Qian, Wang and Wei define ${\rm cod}(\chi)$ to be the {\it codegree} of the irreducible character $\chi$ of $G$ in \cite{Qian1}, although the name codegree of a character was first used by Chillag, Mann, and Manz in \cite{ChillagMann} with a slightly different definition.  The properties of codegrees have gained some interest in recent years.  For example, the codegrees of $p$-groups have been studied in ~\cite{DuLewis, Moreto2022
};
and the codegree analogue of Huppert's $\rho$-$\sigma$ conjecture has been studied in ~
\cite{
YangQian}.

Alizadeh, et al. in \cite{ABGGH} consider finite groups with few codegrees.  In particular, they show that if $|{\rm cod}(G)| = 2$, then $G$ is elementary abelian, and they prove that if $|{\rm cod} (G)| = 3$, then $G$ is solvable.  When a finite group $G$ is such that $|{\rm cod}(G)|\leq 3$, we show that $G$ must be an $M$-group.

\begin{thm}\label{intro-theorem1}
If $G$ is a group with $|{\rm cod} (G)| \leq 3$, then $G$ is 
an $M$-group. 
\end{thm}

We observe that the bound on the hypothesis $|{\rm cod}(G)| \leq 3$ in Theorem \ref{theorem1} cannot be sharpened.  For example, take $G = {\rm SL}_{2}(3)$.  In this case, we have ${\rm cod}(G)=\{1, 3, 4, 12\}$;
however, ${\rm SL}(2, 3)$ is not an $M$-group. 


Let $p$ be a prime, and denote by ${\rm IBr}(G)$ the set of irreducible ($p$-)Brauer characters of $G$.  A ($p$-)Brauer character of $G$ is {\it monomial} if it is induced from a linear ($p$-)Brauer character of some subgroup (not necessarily proper) of $G$.  This definition was introduced by Okuyama in \cite{Okuyama} using module theory.

A group $G$ is called an {\it $M_{p}$-group} if every Brauer character $\varphi\in {\rm IBr}(G)$ is monomial.  Okuyama proves in \cite{Okuyama} that $M_{p}$-groups must be solvable.  It is known using the Fong-Swan theorem that an $M$-group is necessarily an $M_{p}$-group for every prime $p$.  However, an $M_{p}$-group need not be an $M$-group; for example, ${\rm SL}(2, 3)$ is an $M_{2}$-group, not an $M$-group.  In fact, even when $G$ is an $M_p$-group for every prime $p$, it is not necessarily the case that $G$ is an $M$-group.  We next consider the relationship between $M_{p}$-groups and Brauer character degrees.

\begin{thm}\label{theorem2}
Let $G$ be a group and $N$ be a normal subgroup of $G$.
If every nonlinear irreducible $(p-)$Brauer character of $G$ has prime degree, then either $N$ is an $M_{p}$-group or $G/N$ is an $M_{p}$-group.
\end{thm}

Although Tong-Viet \cite{Tong} has given the classification of groups all of whose nonlinear irreducible Brauer characters have prime degrees, the above result is not found in the literature.

\section{Proofs}










We first give a proof of Theorem \ref{intro-theorem1} which we restate here.

\begin{theorem}\label{theorem1}
If $G$ is a group with $|{\rm cod} (G)| \leq 3$, then $G$ is an $M$-group and an $M_{p}$-group for every prime $p$.
\end{theorem}

\begin{proof}
When $|{\rm cod} (G)| = 1$, we know $G = 1$ and there is nothing to prove.

When $|{\rm cod}(G)| = 2$, we have by 
\cite[Lemma 3.1]{ABGGH} $G$ is abelian. It follows 
that $G$ is an $M$-group.

Now, assume that $G$ is a group with $|{\rm cod} (G)| = 3$. Then by \cite[Theorem 3.4]{ABGGH} one of the following is true:

\begin{enumerate}
\item [(i)] $G$ is a $p$-group of nilpotence class $2$ with ${\rm cod}(G)=\{1, p, p^{s}\}$ where $s\geq 2$.
\item [(ii)] $G$ is a Frobenius group with a Frobenius complement of prime order $p$, where $|G|$ is divisible by exactly two primes $p$ and $q$ and ${\rm cod}(G)=\{1, p, q^{s}\}$ for some integer $s\geq 1$.
\end{enumerate}

Using \cite[Theorem 3.5]{ABGGH}, we know that in (ii) that the Frobenius kernel $G'$ is abelian, and in particular, all Sylow subgroups of $G$ in (ii) are abelian, so $G$ is an $M$-group.
%
In (i), we know that $G$ is a $p$-group and so, $G$ is an $M$-group.  Thus, in all cases, $G$ is an $M$-group.

Finally, it follows by the Fong-Swan theorem (see Remark 3.5 (1) of \cite{Okuyama}) 
that $G$ is an $M_{p}$-group for every prime $p$.
\end{proof}




Now, we give the proof of Theorem \ref{theorem2}.

\begin{proof}[Proof of Theorem \ref{theorem2}]
We claim that either $G'N\subseteq PN$ or $N'\subseteq P$, where $G'$ is the derived subgroup of $G$ and $P$ is a Sylow $p$-subgroup of $G$.  Suppose that $G'N \nsubseteq PN$ and we want to prove that $N'\subseteq P$.

If there exists some Brauer character $\theta\in {\rm IBr}(N)$ with $\theta (1) > 1$, then there will exist a Brauer character $\varphi\in {\rm IBr}(G)$ such that $\theta$ is an irreducible constituent of $\varphi_{N}$. Thus, by Clifford's theorem \cite[Corollary 8.7]{Navarro1998}, we have
$$\varphi_{N} = e \sum_{i=1}^{t} \theta_{i},$$
where $\theta_{1} = \theta$ and the $\theta_{i}$ are all conjugate to $\theta$ in $G$ for positive integers $e$ and $t$.  Since $\varphi (1) = et \theta (1) > 1$, we may assume without loss of generality that $\varphi(1) = q$, where $q$ is a prime,and it follows that $\theta (1) = q$ and $e = t = 1$, so that $\varphi_{N}=\theta$.

Thus, for distinct characters $\beta\in {\rm IBr} (G/N)$, we have that $\varphi\beta$ are irreducible and distinct by \cite[Corollary 8.20]{Navarro1998}.  Since $(G/N)' = G'N/N \nsubseteq PN/N$, we conclude by 
\cite[Lemma 2.1]{WCZ} that there exists some nonlinear irreducible Brauer character in ${\rm IBr}(G/N)$.  Then we can choose characters $\beta \in {\rm IBr} (G/N)$ with $\beta (1) > 1$.  Thus,
$$(\beta\varphi)(1)=\beta(1)\varphi(1)=\beta(1)\cdot q>q,$$
which is a contradiction to the hypothesis that every nonlinear irreducible Brauer character of $G$ has prime degree.
Therefore, we conclude that $N$ has no nonlinear irreducible Brauer characters.

Again by 
\cite[Lemma 2.1]{WCZ}, we deduce that the derived subgroup $N'$ of $N$ is contained in a Sylow $p$-subgroup $S$ of $N$.  Then $S$ is a characteristic subgroup of $N$.  Since $N$ is a normal subgroup of $G$, it follows that $N'$ is contained in $P$, as desired.

If $G'N\subseteq PN$, where $P \in {\rm Syl}_{p} (G)$, then $\displaystyle \frac G{PN} \cong \frac {\displaystyle \frac GN}{\displaystyle \frac {PN}N}$ is abelian, and so, it follows by the Fong-Swan theorem that $\displaystyle \frac G{PN}$ is an $M_{p}$-group.  Therefore, $G/N$ is an $M_{p}$-group since $PN/N$ is a normal Sylow $p$-subgroup of $G/N$.  If $N'\subseteq P$, then $N'=N'\cap N\subseteq P\cap N\in {\rm Syl}_{p}(N)$.  Write $P\cap N=S$. Then $N/S$ is abelian and so $N$ is an $M_{p}$-group.
\end{proof}

Under the hypothesis 
of Theorem \ref{theorem2}, we do not necessarily have that $G$ is an $M_{p}$-group.  For example, when $G = {\rm SL} (2, 3)$ and $p = 3$, we see that ${\rm bcd}(G)=\{1, 2, 3\}$, where ${\rm bcd} (G) = \{ \varphi(1) \mid \varphi \in {\rm IBr}(G) \}$, and ${\rm bcd}$ stands for {\it ($p-$)Brauer character degree set}, but ${\rm SL} (2, 3)$ is not an $M_{3}$-group since it has an irreducible $3$-Brauer character of degree $2$.

In addition, we do not even necessarily have that $G$ is solvable. For example,  when $G$ is the alternating group $A_{5}$ and $p=5$, we see that ${\rm bcd} (G) = \{1, 3, 5\}$.  When $|{\rm bcd}(G)| = 2$, however, we have the following.

Suppose $N$ is a normal subgroup of $G$ and consider a character
$\chi \in {\rm Irr} (G)$.  Then $\chi$ is called a {\it relative
$M$-character} with respect to $N$ if there exists a subgroup $H$
with $N\subseteq H\subseteq G$ and a character $\theta\in {\rm
Irr}(H)$ such that $\theta^{G} = \chi$ and $\theta_{N} \in {\rm Irr}
(N)$.  When every character $\chi \in {\rm Irr}(G)$ is a relative
$M$-character with respect to $N$, we say that $G$ is a {\it
relative $M$-group} with respect to $N$.

Observe that $G$ is a relative $M$-group with respect to $1$ if and
only if it is an $M$-group.  Also, if $G$ is a relative $M$-group
with respect to $N$, then $G/N$ is an $M$-group. However, the
converse is not true.  For example, if $G={\rm SL}(2, 3)$ and
$N={\bf Z}(G)$, then $G/N\cong A_{4}$, which is an $M$-group, but
$G$ is not a relative $M$-group with respect to $N$ since $G$ has an
irreducible character of degree $2$ whose restriction to $N$ is
reducible.

Chen \cite{ChenKW} proved that $G$ is a relative $M$-group with
respect to every normal subgroup of $G$ if $G$ is either metabelian
(i.e. $G'$ is abelian) or an $M$-group $G$ with ${\rm cd}(G)=\{1, u,
v\}$, where ${\rm cd}(G)=\{\chi(1)\mid \chi\in {\rm Irr}(G)\}$ and
$u, v$ are different primes.

\begin{prop}
Let $G$ be a group and $p$ be a prime.
If ${\rm bcd}(G) = \{ 1, m \}$, where $m>1$ and $p$ does not divide $m$, then $G$ is an $M_{p}$-group. In particular, $G$ is solvable.
\end{prop}

\begin{proof}
Since $p$ does not divide the degree of any irreducible Brauer character of $G$, it follows by It\^o-Michler theorem that $G$ has a normal Sylow $p$-subgroup $P$ and
$${\rm cd} (G/P) = {\rm bcd} (G/P) = {\rm bcd} (G) = \{1, m\}.$$
Therefore, $G/P$ is metabelian by \cite[Corollary 12.6]{Isaacs1976}, and $G/P$ is an $M$-group by 
\cite[Theorem 1]{ChenKW}.  It follows that $G/P$ is an $M_{p}$-group.
Since $P$ is a normal Sylow $p$-subgroup of $G$, we conclude that $G$ is an $M_{p}$-group.
\end{proof}

If $p$ divides $m$, then we do not necessarily have that $G$ is an $M_{p}$-group.
For example, if $G$ is the symmetric $S_{5}$ on five letters and $p=2$,
then ${\rm bcd}(S_{5})=\{1, 4\}$, however, $S_{5}$ is not an $M_{2}$-group.


\section*{Acknowledgments}

The first author thanks support of China Scholarship Council and the
International Training Program of Henan Province,
the program of Henan University of Technology (2024PYJH019),
the projects of Education Department of Henan Province (23A110010, YJS2022JC16),
and the programs of Henan Province (HNGD2024020, 242300421384).



\end{document}